\documentclass[12pt]{article}
\pdfoutput=1
\usepackage{geometry}
\usepackage{array}
\usepackage{mathtools}
\usepackage{amsmath}
\usepackage{amsthm}
\usepackage{amssymb}
\usepackage{diagbox}

\makeatletter

\providecommand{\tabularnewline}{\\}

\numberwithin{equation}{section}
\theoremstyle{plain}
\newtheorem{thm}{\protect\theoremname}[section]
\theoremstyle{plain}
\newtheorem{prop}[thm]{\protect\propositionname}
\theoremstyle{plain}
\newtheorem{lem}[thm]{\protect\lemmaname}
\theoremstyle{remark}
\newtheorem{rem}[thm]{\protect\remarkname}
\theoremstyle{plain}
\newtheorem{cor}[thm]{\protect\corollaryname}

\usepackage{bm}
\usepackage{tikz}
\usetikzlibrary{arrows}
\usepackage{hyperref}
\usepackage{colortbl}
\usepackage{xparse}
\usepackage{caption}
\usepackage{microtype}

\usepackage{titlesec}
\titleformat{\section}
  {\normalfont\large\bfseries}{\thesection.}{.8ex}{} %changed from Large
\titleformat{\subsection}
  {\normalfont\bfseries}{\thesubsection.}{.8ex}{} %large removed
\usetikzlibrary{shapes}
\tikzstyle{pathdefault}=[draw, line width=1, solid, color=black]
\tikzstyle{nodedefault}=[circle, inner sep=1.1, fill=black]
\tikzstyle{empty}=[]
\tikzstyle{nodeellipsis}=[circle, inner sep=0.5, fill=black]
\tikzstyle{pathcolor1}=[draw, line width=1.3, densely dashed, color=red]
\tikzstyle{pathcolor2}=[draw, line width=1.6, densely dotted, color=blue]
\tikzstyle{pathcolorlight}=[draw, line width=1, dotted, color=lightgray]
\tikzstyle{arbpathcolor0}=[line width=1, dashdotted, color=black]
\tikzstyle{arbpathcolor1}=[line width=1, densely dashed, color=red]
\tikzstyle{arbpathdefault}=[line width=1, densely dotted, color=blue]

\newcounter{id}
\newcommand{\drawlinedotswithstyle}[4]{
 \def\x{{#3}}
 \def\y{{#4}}
 \tikzstyle{thispathstyle}=[#1]
 \tikzstyle{thisnodestyle}=[#2]
 \setcounter{id}{-1} %start id at -1
 \foreach \j in {#3}{\stepcounter{id}} %id is one less than num of pts
 \foreach \i in {1,...,\the\value{id}}{  %loop through later indices
  \path[thispathstyle] (\x[\i],\y[\i]) --(\x[\i-1],\y[\i-1]); %draw edge
 }
 \foreach \i in {1,...,\the\value{id}}{  %loop through later indices
  \node[thisnodestyle] at (\x[\i],\y[\i]) {}; %draw node
 }
 \node[thisnodestyle] at (\x[0],\y[0]) {}; %draw first node outside of loop
}

\DeclareDocumentCommand{\drawlinedots}{ O{pathdefault} O{nodedefault} m m}{\drawlinedotswithstyle{#1}{#2}{#3}{#4}}
\DeclareDocumentCommand{\drawlinedotsred}{ O{pathcolor1} O{nodedefault} m m}{\drawlinedotswithstyle{#1}{#2}{#3}{#4}}

\let\originalleft\left
\let\originalright\right
\renewcommand{\left}{\mathopen{}\mathclose\bgroup\originalleft}
\renewcommand{\right}{\aftergroup\egroup\originalright}

\newcommand{\leqnomode}{\tagsleft@true\let\veqno\@@leqno}
\newcommand{\reqnomode}{\tagsleft@false\let\veqno\@@eqno}
\reqnomode

\makeatother

\usepackage{babel}
\providecommand{\corollaryname}{Corollary}
\providecommand{\lemmaname}{Lemma}
\providecommand{\propositionname}{Proposition}
\providecommand{\remarkname}{Remark}
\providecommand{\theoremname}{Theorem}

\usepackage{authblk}
\title{A subfamily of skew Dyck paths related to $k$-ary trees}
\author[1]{Yuxuan Zhang\thanks{\tt{yz885@duke.edu}}} 
\author[2]{Yan Zhuang\thanks{\tt{yazhuang@davidson.edu}}}
\affil[1]{Department of Statistical Science, Duke University} 
\affil[2]{Department of Mathematics and Computer Science, Davidson College}
\date{\today}

\begin{document}

\maketitle
\begin{abstract}
We introduce a subfamily of skew Dyck paths called \textit{box paths} and show that they are in bijection with pairs of ternary trees, confirming an observation stated previously on the On-Line Encyclopedia of Integer Sequences. More generally, we define \textit{$k$-box paths}, which are in bijection with $(k+1)$-tuples of $(k+2)$-ary trees. A bijection is given between $k$-box paths and a subfamily of $k_{t}$-Dyck paths, as well as a bijection with a subfamily of $(k,\ell)$-threshold sequences. We also study the refined enumeration of $k$-box paths by the number of returns and the number of long ascents. Notably, the distribution of long ascents over $k$-box paths generalizes the Narayana distribution on Dyck paths, and we find that $(k-3)$-box paths with exactly two long ascents provide a combinatorial model for the second $k$-gonal numbers.
\end{abstract}
\textbf{\small{}Keywords: }{\small{}skew Dyck paths, $k$-ary trees, $k_{t}$-Dyck paths, threshold sequences, Narayana numbers, second $k$-gonal numbers
}{\let\thefootnote\relax}

\section{Introduction}

A \textit{lattice path} is a path in the integer lattice $\mathbb{Z}^{n}$ consisting of a sequence of steps from a prescribed step set and satisfying prescribed restrictions. Perhaps the most well-studied among classical lattice paths are Dyck paths. A \textit{Dyck path} of \textit{semilength} $n$ is a path in $\mathbb{Z}^{2}$\textemdash consisting of $2n$ steps from the step set $\{(1,1),(1,-1)\}$\textemdash which starts at the origin, ends on the $x$-axis, and never traverses below the $x$-axis. The steps $(1,1)$ are called \textit{up steps} and are represented by the letter $U$, and the $(1,-1)$ are called \textit{down steps} and represented by $D$. It is well known that the number of Dyck paths of semilength $n$ is the $n$th Catalan number $C_{n}={2n \choose n}/(n+1)$.

In 2010, Deutsch, Munarini, and Rinaldi \cite{Deutsch2010} introduced a generalization of Dyck paths called \textit{skew Dyck paths}, which are defined in the same way except that we also allow \textit{left steps} $L=(-1,-1)$ as long as up steps never intersect with left steps. (In other words, we cannot have an up step followed by a left step, or a left step followed by an up step.) Thus, Dyck paths are precisely skew Dyck paths without left steps. 

Lattice paths with step set $S$ can be represented as words on the alphabet $S$ satisfying appropriate restrictions. In particular, we represent skew Dyck paths as words $w$ over the alphabet $\{U,D,L\}$ such that $w$ avoids the contiguous subwords $UL$ and $LU$, the number of $U$s in $w$ is equal to the total number of $D$s and $L$s, and the total number of $D$s and $L$s in any prefix of $w$ is never more than the number of $U$s in that prefix. See Figure \ref{f-sdp} for an example. We use superscript notation to represent consecutive occurrences of the same step, e.g., $U^{3}=UUU$. We refer to contiguous subwords as \textit{factors}, and whenever we refer to a factor in a lattice path, we really mean a factor inside the corresponding word.
\begin{figure}
\noindent \begin{centering}
\begin{tikzpicture}[scale=0.6] 
\draw [line width=0] (4,2); 
\draw[pathcolorlight] (0,0) -- (9,0); 
\drawlinedots{0,1,2,3,4,5,6,7,8,7,6,7,8,9,8}{0,1,0,1,2,3,4,5,4,3,2,1,2,1,0}
\end{tikzpicture}
\hspace{0.5cm}
\par\end{centering}
\caption{\label{f-sdp}The skew Dyck path corresponding to the word $UDU^{5}DL^{2}DUDL$}
\end{figure}

Within lattice path enumeration, there is interest in counting paths by the number of occurrences of certain factors; see \cite{Deutsch1999} for an early work on this topic. For example, the number of Dyck paths of semilength $n$ with exactly $j$ $UD$-factors (a.k.a. \textit{peaks}) is equal to the \textit{Narayana number}
\[
N_{n,j}\coloneqq\frac{1}{n}{n \choose j}{n \choose j-1}=\frac{1}{j}{n-1 \choose j-1}{n \choose j-1}.
\]

Our present work concerns $UDL$-factors in skew Dyck paths. The array of numbers counting skew Dyck paths by semilength and number of $UDL$-factors is given in OEIS entry A128728 \cite{oeis}. Recently, Prodinger \cite{Prodinger} derived the functional equation
\begin{equation}
x^{2}R^{3}-x(2-x)R^{2}+(1-x^{2})R-1+x+x^{2}-tx^{2}=0\label{e-prod}
\end{equation}
for the bivariate generating function $R=R(t,x)$ whose coefficient of $t^{k}x^{n}$ is the number of skew Dyck paths of semilength $n$ with $k$ $UDL$s. With a computer algebra system such as Maple, one may use Equation (\ref{e-prod}) to compute a (truncated) series expansion for $R$, from which it appears that the coefficient of $t^{n}x^{3n-1}$ is equal to ${3n-1 \choose n}/(3n-1)$ for all $n\geq1$. This observation had previously been reported on the OEIS page for A128728 with the following comment:
\begin{center}
\texttt{Apparently, T(3k-1,k) = binomial(3k-1,k)/(3k-1).}
\par\end{center}

The starting point of this paper was the first author's undergraduate honors thesis at Davidson College (supervised by the second author), which (among other things) sought to provide an explanation for this observation. It can be shown that among skew Dyck paths with exactly $n$ $UDL$s, the shortest paths have semilength $3n-1$. We call such paths \textit{box}\textit{ paths} of \textit{size} $n$. (Visually, a $UDL$-factor in a skew Dyck path is reminiscient of a box protruding from a down-slope of the path.) See Figure \ref{f-box3} for all box paths of size 3. 

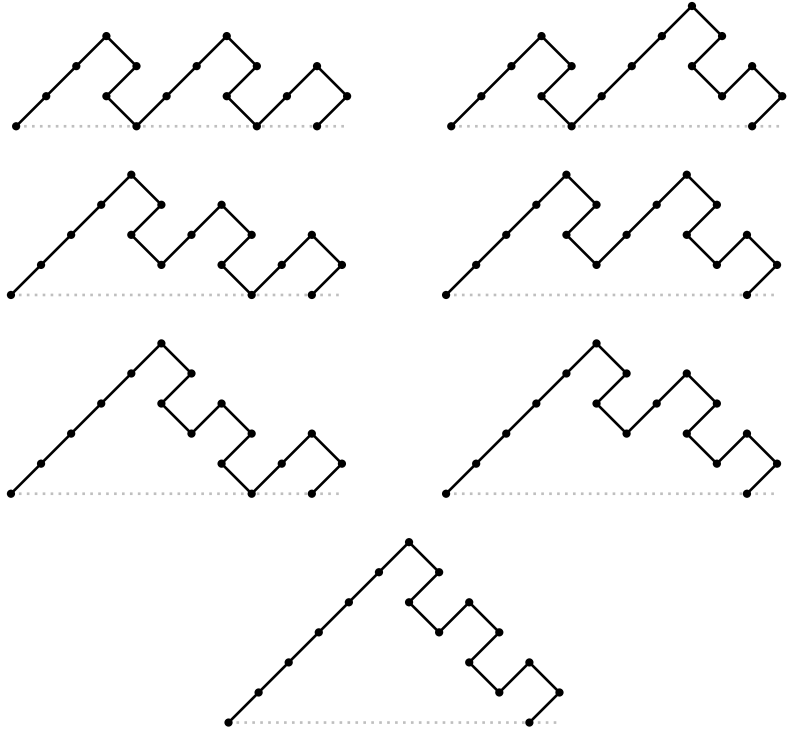
\begin{figure}
\noindent \begin{centering}
\begin{tikzpicture}[scale=0.4] 
\draw [line width=0] (4,2); 
\draw[pathcolorlight] (0,0) -- (11,0); 
\drawlinedots{0,1,2,3,4,3,4,5,6,7,8,7,8,9,10,11,10}{0,1,2,3,2,1,0,1,2,3,2,1,0,1,2,1,0}
\end{tikzpicture}
\hspace{1cm}
\begin{tikzpicture}[scale=0.4] 
\draw [line width=0] (4,2); 
\draw[pathcolorlight] (0,0) -- (11,0); 
\drawlinedots{0,1,2,3,4,3,4,5,6,7,8,9,8,9,10,11,10}{0,1,2,3,2,1,0,1,2,3,4,3,2,1,2,1,0}
\end{tikzpicture}

\vspace{0.5cm}

\begin{tikzpicture}[scale=0.4] 
\draw[pathcolorlight] (0,0) -- (11,0); 
\drawlinedots{0,1,2,3,4,5,4,5,6,7,8,7,8,9,10,11,10}{0,1,2,3,4,3,2,1,2,3,2,1,0,1,2,1,0}
\end{tikzpicture}
\hspace{1cm}
\begin{tikzpicture}[scale=0.4] 
\draw [line width=0] (4,2); 
\draw[pathcolorlight] (0,0) -- (11,0); 
\drawlinedots{0,1,2,3,4,5,4,5,6,7,8,9,8,9,10,11,10}{0,1,2,3,4,3,2,1,2,3,4,3,2,1,2,1,0}
\end{tikzpicture}

\vspace{0.5cm}

\begin{tikzpicture}[scale=0.4] 
\draw[pathcolorlight] (0,0) -- (11,0); 
\drawlinedots{0,1,2,3,4,5,6,5,6,7,8,7,8,9,10,11,10}{0,1,2,3,4,5,4,3,2,3,2,1,0,1,2,1,0}
\end{tikzpicture}
\hspace{1cm}
\begin{tikzpicture}[scale=0.4] 
\draw[pathcolorlight] (0,0) -- (11,0); 
\drawlinedots{0,1,2,3,4,5,6,5,6,7,8,9,8,9,10,11,10}{0,1,2,3,4,5,4,3,2,3,4,3,2,1,2,1,0}
\end{tikzpicture}

\vspace{0.5cm}

\begin{tikzpicture}[scale=0.4] 
\draw[pathcolorlight] (0,0) -- (11,0); 
\drawlinedots{0,1,2,3,4,5,6,7,6,7,8,9,8,9,10,11,10}{0,1,2,3,4,5,6,5,4,3,4,3,2,1,2,1,0}
\end{tikzpicture}
\par\end{centering}
\caption{\label{f-box3}All box paths of size 3}

\end{figure}

We wish to prove that there are ${3n-1 \choose n}/(3n-1)$ box paths of size $n$. The numbers ${3n-1 \choose n}/(3n-1)$ have several established combinatorial interpretations, including pairs of ternary trees with $n-1$ total nodes. Thus, it suffices to show that box paths and pairs of ternary trees are in bijection. 

More generally, define a \textit{$k$-box path} of \textit{size} $n$ to be a skew Dyck path of semilength $(k+2)n-1$ with exactly $n$ $UD^{k}L$-factors (so that box paths are 1-box paths). We give a recursive bijection between $k$-box paths and $(k+1)$-tuples of $(k+2)$-ary trees; this reduces to the above result for $k=1$. While one cannot have a $UL$-factor in a skew Dyck path, we describe a way to make sense of a ``0-box path'' as being essentially a Dyck path, so the case $k=0$ of our result recovers the well-known fact that Dyck paths are in bijection with binary trees. In this way, we can think of $k$-box paths as being a generalization of Dyck paths. 

This paper is organized as follows. We begin with some preliminary definitions in Section \ref{sec-2}. Next, we show in Section \ref{sec-3} that $k$-box paths of size $n$ are in bijection with $(k+1)$-tuples of $(k+2)$-ary trees with $n-1$ total nodes, the number of which is given by the formula
\[
\frac{1}{(k+2)n-1}{(k+2)n-1 \choose n}.
\]
These numbers also count $(k+1)_{k}$-Dyck paths of size $n-1$ and $(k+2,k)$-threshold sequences of length $n-1$, which are subfamilies of the $k_{t}$-Dyck paths introduced by Selkirk \cite{Selkirk2019} and of the $(k,\ell)$-threshold sequences introduced by Rusu \cite{Rusu2022}, respectively. In Section \ref{s-selkirk}, we give a simple bijection between $k$-box paths and $(k+1)_{k}$-Dyck paths, and in Section \ref{s-rusu}, we give one between $k$-box paths and $(k+2,k)$-threshold sequences. Finally, in Sections \ref{sec-6}\textendash \ref{sec-7}, we study the distribution of two statistics over $k$-box paths\textemdash the number of returns and the number of long ascents\textemdash and give an exact enumeration formula for each. The latter distribution generalizes the Narayana distribution on Dyck paths, and curiously, $k$-box paths with exactly two long ascents are shown to be counted by the second $(k+3)$-gonal numbers.

\section{Preliminaries: \texorpdfstring{$k$}{k}-ary trees and \texorpdfstring{$k$}{k}-Dyck paths} \label{sec-2}

A \textit{ternary tree} is a rooted tree in which each node has at most three children, which are distinguished as the left child, the middle child, and the right child; see Figure \ref{f-ternary} for an example. More generally, a $k$\textit{-ary tree} is a tree in which each node has at most $k$ children, all of which are distinguished in the analogous way. 
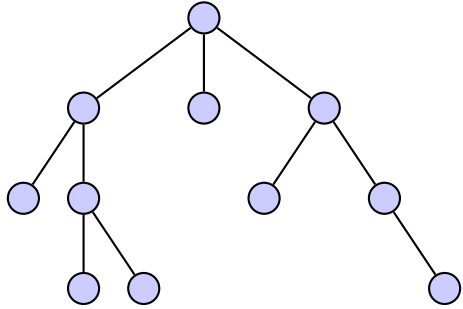
\begin{figure}
\noindent \begin{centering}
\begin{center}
\begin{tikzpicture}[scale=0.4,auto,   thick,root/.style={circle,fill=blue!20,draw,font=\sffamily,minimum size=1em},left/.style={circle,fill=blue!20,draw,font=\sffamily,minimum size=1em},mid/.style={circle,fill=blue!20,draw,font=\sffamily,minimum size=1em},right/.style={circle,fill=blue!20,draw,font=\sffamily,minimum size=1em}]

\node[root] (1) at (10,10) {};
\node[left] (2) at (6,7) {};
\node[mid] (3) at (10,7) {};
\node[right] (4) at (14,7) {};
\node[left] (5) at (4,4) {};
\node[left] (6) at (6,4) {};
\node[right] (7) at (12,4) {};
\node[right] (8) at (16,4) {};
\node[left] (9) at (6,1) {};
\node[left] (10) at (8,1) {};
\node[right] (11) at (18,1) {};

\foreach \from/\to in {1/2,1/3,1/4,2/5,2/6,4/7,4/8,6/9,6/10,8/11}
\draw (\from) -- (\to);

\end{tikzpicture}
\end{center}
\par\end{centering}
\caption{\label{f-ternary}A ternary tree with 11 nodes}
\end{figure}

Let ${\cal T}_{n}^{(k)}$ be the set of $k$-ary trees with $n$ nodes, and let 
\[
C_{k}=C_{k}(x)\coloneqq\sum_{n=0}^{\infty}\left|{\cal T}_{n}^{(k)}\right|x^{n}
\]
be the ordinary generating function for $k$-ary trees. Note that a $k$-ary tree is either empty or can be recursively decomposed as a $k$-tuple of (possibly empty) subtrees, all of which are attached to the root. This leads to the functional equation
\[
C_{k}=1+xC_{k}^{k},
\]
from which the formulas
\begin{equation}
\left|{\cal T}_{n}^{(k)}\right|=[x^{n}]\,C_{k}=\frac{1}{n}{kn \choose n-1}=\frac{1}{kn+1}{kn+1 \choose n}\label{e-gencat}
\end{equation}
and 
\begin{equation}
[x^{n}]\,C_{k}^{r}=\frac{r}{n}{kn+r-1 \choose n-1}=\frac{r}{kn+r}{kn+r \choose n}\label{e-rthpower}
\end{equation}
follow from a straightforward application of Lagrange inversion \cite{Gessel2016}. Note that the numbers in (\ref{e-rthpower}) count $r$-tuples of $k$-ary trees with a total of $n$ nodes, and these numbers are called
(\textit{generalized}) \textit{Fuss\textendash Catalan} \textit{numbers} and specialize to the Catalan numbers at $r=1$ and $k=2$.

To make the connection between our skew Dyck paths and $k$-ary trees, it is easier to work with another family of lattice paths which are in bijection with $k$-ary trees. Given a positive integer $k$, a $k$-\textit{Dyck path} of \textit{size} $n$ is a path in $\mathbb{Z}^{2}$\textemdash with steps $U=(1,1)$ and $D^{(k)}=(k,-k)$\textemdash that begins at the origin $(0,0)$, ends at $(2kn,0)$, and never traverses below the $x$-axis. Notice that we recover ordinary Dyck paths by setting $k=1$, and that the size of a $k$-Dyck path is equal to its number of $D^{(k)}$ steps. 

Every $k$-Dyck path $\mu$ is either the empty path or can be uniquely decomposed in the form 
\[
\mu=U\mu_{1}U\mu_{2}U\mu_{3}\cdots U\mu_{k}D^{(k)}\mu_{k+1}
\]
where $\mu_{1},\dots,\mu_{k+1}$ are (possibly empty) $k$-Dyck paths; this is a generalization of the ``first return decomposition'' for ordinary Dyck paths. Because the decomposition of $k$-ary trees into subtrees is of the same form, we have a recursive bijection between $k$-Dyck paths and $(k+1)$-ary trees.

Given $k\geq2$, let us call a skew Dyck path an \textit{augmented k-Dyck path} of \textit{size} $n$ if it can be obtained by taking a $k$-Dyck path of size $n$ and replacing each $D^{(k)}$ step with a factor $UD^{k-1}LD$. See Figure \ref{f-dyckaug} for an example. Clearly, an augmented $k$-Dyck path cannot have any $D$ or $L$ steps except within a $UD^{k-1}LD$-factor, so augmented $k$-Dyck paths of size $n$ are precisely the skew Dyck paths of the form 
\[
U^{a_{1}}D^{k-1}LDU^{a_{2}}D^{k-1}LD\cdots U^{a_{n}}D^{k-1}LD
\]
where $a_{1},a_{2},\dots,a_{n}$ are positive integers. Let ${\cal \hat{D}}_{n}^{(k)}$ be the set of augmented $k$-Dyck paths of size $n$, and ${\cal \hat{D}}^{(k)}\coloneqq\bigsqcup_{n\geq0}{\cal \hat{D}}_{n}^{(k)}$.
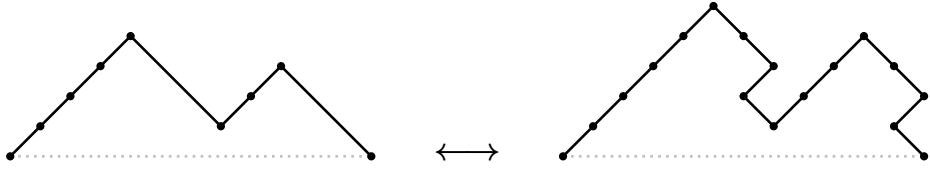
\begin{figure}
\noindent \begin{centering}
\begin{tikzpicture}[scale=0.4] 
\draw [line width=0] (4,2); 
\draw[pathcolorlight] (0,0) -- (12,0); 
\drawlinedots{0,1,2,3,4,7,8,9,12}{0,1,2,3,4,1,2,3,0}
\end{tikzpicture}
\hspace{0.5cm}
$\boldsymbol{\longleftrightarrow}$
\hspace{0.5cm}
\begin{tikzpicture}[scale=0.4] 
\draw [line width=0] (4,2); 
\draw[pathcolorlight] (0,0) -- (12,0); 
\drawlinedots{0,1,2,3,4,5,6,7,6,7,8,9,10,11,12,11,12}{0,1,2,3,4,5,4,3,2,1,2,3,4,3,2,1,0}
\end{tikzpicture}
\par\end{centering}
\caption{\label{f-dyckaug}A $3$-Dyck path and the corresponding augmented
$3$-Dyck path}
\end{figure}

\section{\texorpdfstring{$k$}{k}-box paths and their enumeration} \label{sec-3}

Recall that we defined a box path of size $n$ to be a skew Dyck path of semilength $3n-1$ with exactly $n$ $UDL$-factors. More generally, define a \textit{$k$-box path} of \textit{size} $n$ to be a skew Dyck path of semilength $(k+2)n-1$ with exactly $n$ $UD^{k}L$-factors. The following proposition shows that $k$-box paths of size $n$ are the shortest skew Dyck paths with $n$ $UD^{k}L$-factors.

\begin{prop}
\label{p-UDLmin}
Among skew Dyck paths with exactly $n$ $UD^{k}L$-factors, the paths of smallest semilength are of semilength $(k+2)n-1$.
\end{prop}

\begin{proof}
Since
\[
\underset{n-1\:\mathrm{times}}{\underbrace{U^{k+2}D^{k}LD\cdots U^{k+2}D^{k}LD}}U^{k+1}D^{k}L
\]
is a skew Dyck path of semilength $(k+2)n-1$ with $n$ $UD^{k}L$-factors, it remains to show that there cannot be a shorter one. Consider an arbitrary skew Dyck path with exactly $n$ $UD^{k}L$-factors; note that each $UD^{k}L$-factor except the last one must be part of a $UD^{k}L^{j}D$-factor for some positive integer $j$, as each $L$ step (except the ending step) must be followed by another $L$ step or a $D$ step. This means that we have $n-1$ factors of the form $UD^{k}L^{j}D$ along with a $UD^{k}L$-factor. Note that each factor of the form $UD^{k}L^{j}D$ contributes $k+j+1$ to the semilength, whereas the ending $UD^{k}L$-factor contributes $k+1$ to the semilength. And since each $j$ must be at least 1, the semilength must be at least $(k+2)(n-1)+k+1=(k+2)n-1$ as desired.
\end{proof}
Let $\mathcal{B}_{n}^{(k)}$ denote the set of $k$-box paths of size $n$, and $\mathcal{B}^{(k)}\coloneqq\bigsqcup_{n\geq0}\mathcal{B}_{n}^{(k)}$ the set of all $k$-box paths.

\begin{lem}
\label{l-tboxstruct}
Every $k$-box path $\mu$ of size $n$ is of
the form 
\[
U^{a_{1}}D^{k}LD\cdots U^{a_{n-1}}D^{k}LDU^{a_{n}}D^{k}L
\]
where the $a_{1},a_{2},\dots,a_{n}$ are positive integers. Moreover, as long as $a_{1}+a_{2}+\cdots+a_{n}=(k+2)n-1$ and no prefix of $\mu$ contains fewer $U$s than the total number of $D$s and $L$s, then any $\mu$ of the form 
\[
U^{a_{1}}D^{k}LD\cdots U^{a_{n-1}}D^{k}LDU^{a_{n}}D^{k}L
\]
is a $k$-box path of size $n$.
\end{lem}

\begin{proof}
This essentially follows from the same reasoning as the proof of Proposition \ref{p-UDLmin}. Let $\mu$ be a $k$-box path of size $n$, so that $\mu$ has semilength $(k+2)n-1$ and has $n$ $UD^{k}L$-factors. As explained in the proof of Proposition \ref{p-UDLmin}, each of these $UD^{k}L$-factors except the last one must be immediately followed by a $D$ step, so we have $n-1$ factors of the form $UD^{k}LD$ along with a $UD^{k}L$-factor. Together, these factors contribute the full $(k+2)n-1$ to the semilength, so $\mu$ cannot have any additional $D$ steps or $L$ steps. Therefore, $\mu$ must be of the desired form.

Now, assume that $\mu=U^{a_{1}}D^{k}LD\cdots U^{a_{n-1}}D^{k}LDU^{a_{n}}D^{k}L$ satisfies $a_{1}+a_{2}+\cdots+a_{n}=(k+2)n-1$ and the given prefix condition. These two conditions ensure that $\mu$ is a skew Dyck path of semilength $(k+2)n-1$, and since $\mu$ contains exactly $n$ $UD^{k}L$-factors, it follows that $\mu$ is a $k$-box path of size $n$.
\end{proof}

\begin{rem}
From Lemma \ref{l-tboxstruct} and the correspondence between $k$-Dyck paths and augmented $k$-Dyck paths, we see that $k$-box paths are in bijection with $(k+1)$-Dyck paths\textemdash which have steps $U=(1,1)$ and $D^{(k+1)}=(k+1,-k-1)$\textemdash but modified so that they end with a special down step $D^{(k)}=(k,-k)$. Equivalently, these are $(k+1)$-Dyck paths that end at the line $y=k$ as opposed to the $x$-axis.
\end{rem}

Lemma \ref{l-tboxstruct} allows us to establish a decomposition of $k$-box paths in terms of augmented $(k+1)$-Dyck paths. To do so, we need the notion of a \textit{return} in a skew Dyck path, which is a point on the path where a step ends at the $x$-axis. Revisiting Figure \ref{f-box3}, we see that among box paths of size 3, there are 3 paths with 1 return, 3 paths with 2 returns, and 1 path with 3 returns. More generally, a \textit{return to the line} $y=a$ is a point where a down step or a left step ends at the line $y=a$.

\begin{lem}
\label{l-tboxdecomp}
We have
\[
\mathcal{B}^{(k)}=\{\,\mu_{1}U\mu_{2}U\cdots\mu_{k+1}UD^{k}L:\mu_{1},\mu_{2},\dots,\mu_{k+1}\in{\cal \hat{D}}^{(k+1)}\,\}.
\]
\end{lem}

\begin{proof}
Let $\mu\in\mathcal{B}^{(k)}$. Then we can write $\mu$ in the desired form in the following way. Let $\mu_{1}$ be the prefix of $\mu$ up to the penultimate return to the $x$-axis. (If $\mu$ has only one return, then let $\mu_{1}$ be empty.) Then, ignoring the next $U$, let $\mu_{2}$ be the prefix of the remaining segment of the path up to the penultimate return to the line $y=1$ (again setting $\mu_{2}$ to be empty if no such return exists). Then $\mu_{3},\dots,\mu_{k+1}$ are defined in the analogous way. In light of Lemma~\ref{l-tboxstruct}, we see that each $\mu_{i}$ must be a skew Dyck path of the form 
\[
U^{b_{1}}D^{k}LDU^{b_{2}}D^{k}LD\cdots U^{b_{m}}D^{k}LD
\]
and is therefore an augmented $(k+1)$-Dyck path, as desired.

Conversely, if each $\mu_{i}$ is an augmented $(k+1)$-Dyck path, then they must each be of the form 
\[
U^{b_{1}}D^{k}LDU^{b_{2}}D^{k}LD\cdots U^{b_{m}}D^{k}LD.
\]
Then $\mu=\mu_{1}U\mu_{2}U\cdots\mu_{k+1}UD^{k}L$ is of the form
\[
U^{a_{1}}D^{k}LD\cdots U^{a_{n-1}}D^{k}LDU^{a_{n}}D^{k}L;
\]
moreover, the equation $a_{1}+a_{2}+\cdots+a_{n}=(k+2)n-1$ and the prefix condition of Lemma~\ref{l-tboxstruct} follow from the decomposition $\mu=\mu_{1}U\mu_{2}U\cdots\mu_{k+1}UD^{k}L$ and the fact that each $\mu_{i}$ is a skew Dyck path. Then it follows from Lemma \ref{l-tboxstruct} that $\mu$ is a $k$-box path.
\end{proof}

\begin{thm}
\label{t-kbox}
The number of $k$-box paths of size $n$ is equal to the number of $(k+1)$-tuples of $(k+2)$-ary trees with a total of $n-1$ nodes, which is
\[
f_{k,n}\coloneqq\frac{k+1}{n-1}{(k+2)n-2 \choose n-2}=\frac{1}{(k+2)n-1}{(k+2)n-1 \choose n}.
\]
\end{thm}

\begin{proof}
By Lemma \ref{l-tboxdecomp}, we can uniquely decompose a $k$-box path of size $n$ as a sequence of $k+1$ augmented $(k+1)$-Dyck paths with total size $n-1$. These are in bijection with $(k+1)$-tuples of $(k+1)$-Dyck paths with total size $n-1$ (we replace each $UD^{k}LD$-factor with a $D^{(k+1)}$ step), and such paths are in bijection with $(k+2)$-ary trees of the same size. The formula $\frac{k+1}{n-1}{(k+2)n-2 \choose n-2}$ then comes directly from (\ref{e-rthpower}), which can be shown to be equal to $\frac{1}{(k+2)n-1}{(k+2)n-1 \choose n}$ by routine algebraic manipulation.
\end{proof}

By taking $k=1$, we get that there are ${3n-1 \choose n}/(3n-1)$ box paths of size $n$, thus confirming the initial observation which motivated this work.

\begin{rem}
\label{r-k=00003D0}
Our definition of a $k$-box path is only valid for positive integers $k$ and not for $k=0$; after all, a skew Dyck path cannot have a $UL$-factor as the $U$ and $L$ steps would overlap. Nonetheless, let us examine what would happen if $UL$-factors were allowed. A ``$0$-box path'' would be of the form $\mu UL$ where $\mu$ is an ``augmented 1-Dyck path''\textemdash i.e., the result of taking a Dyck path and replacing each $D$ step with a $ULD$. If we delete each $UL$-factor\textemdash that is, we can view the $U$ and $L$ steps as ``canceling out''\textemdash then the result would be a Dyck path of semilength $n-1$ (where $n$ is the number of $UL$-factors). Conversely, given a Dyck path of semilength $n$, we can replace each $D$ step with a $ULD$-factor and append a $UL$-factor to get a ``$0$-box path'' with $n$ $UL$-factors. Therefore, we can think of a ``$0$-box path'' of size $n$ to be a Dyck path of semilength $n-1$. Then Theorem \ref{t-kbox} recovers the well-known fact that Dyck paths are in bijection with binary trees. Using this interpretation, all of our results later in this paper for $k$-box paths become valid for $k=0$ as well.
\end{rem}

Before proceeding, we note that there is a natural subfamily of $k$-box paths which are equinumerous to (single) $(k+2)$-ary trees, namely, those that end with a $U^{k+1}D^{k}L$-factor. Let us call such paths \textit{tailed $k$-box paths}. Note that all $0$-box paths are tailed.

\begin{prop}
\label{p-tailed}
The number of tailed $k$-box paths of size $n$ is equal to the number of $(k+2)$-ary trees with $n-1$ nodes, which is
\[
\frac{1}{n-1}{(k+2)(n-1) \choose n-2}=\frac{1}{(k+2)(n-1)+1}{(k+2)(n-1)+1 \choose n-1}.
\]
\end{prop}

\begin{proof}
Tailed $k$-box paths are precisely the $k$-box paths where $\mu_{2},\dots,\mu_{k+1}$ are all empty in the decomposition given by Lemma \ref{l-tboxdecomp}. Consequently, tailed $k$-box paths of size $n$ are in bijection with augmented $(k+1)$-Dyck paths of size $n-1$, which are in bijection with $(k+1)$-Dyck paths of size $n-1$ and thus $(k+2)$-ary trees with $n-1$ nodes.
\end{proof}

By dividing the formula in Proposition \ref{p-tailed} by that in Theorem \ref{t-kbox} and simplifying, we find that the proportion of tailed $k$-box paths among paths in $\mathcal{B}_{n}^{(k)}$ is equal to
\[
\frac{((k+1)n)_{k}}{(k+1)((k+2)n-2)_{k}}
\]
where $(x)_{j}$ denotes the falling factorial $(x)_{j}\coloneqq x(x-1)\cdots(x-j+1)$. When $k=0$, this proportion is always equal to 1, which makes sense in light of Remark \ref{r-k=00003D0} because every ``$0$-box path'' ends with a $UL$-factor. For general $k$, this proportion tends to
\[
\lim_{n\rightarrow\infty}\frac{((k+1)n)_{k}}{(k+1)((k+2)n-2)_{k}}=\frac{(k+1)^{k-1}}{(k+2)^{k}}
\]
as $n\rightarrow\infty$. In particular, by taking $k=1$, we see that approximately $1/3$ of box paths of size $n$ end with a $UUDL$ when $n$ is large.

\section{\label{s-selkirk}A bijection between \texorpdfstring{$k$}{k}-box paths and \texorpdfstring{$(k+1)_{k}$}{(k+1)k}-Dyck paths}

Recall that a $k$-\textit{Dyck path} of \textit{size} $n$ is a path in $\mathbb{Z}^{2}$ which begins at the origin $(0,0)$, ends at $(2kn,0)$, never traverses below the $x$-axis, and consists of steps $U=(1,1)$ and $D^{(k)}=(k,-k)$. A $k_{t}$\textit{-Dyck path} of size $n$ (where $t$ is a non-negative integer) is defined in the same way except that the path never traverses below the line $y=-t$ (as opposed to the $x$-axis), so that we recover $k$-Dyck paths upon setting $t=0$. These generalizations of $k$-Dyck paths were introduced recently by Selkirk \cite{Selkirk2019} and further studied in \cite{Asinowski2022,Heuberger2023}.

Selkirk found that, for $0\leq t\leq k-1$, the number of $k_{t}$-Dyck paths of size $n$ is given by the formula 
\[
\frac{t+1}{(k+1)n+t+1}{(k+1)n+t+1 \choose n}
\]
\cite[Proposition 2.2.2]{Selkirk2019}; comparing with (\ref{e-rthpower}), we see that this is the number of $(t+1)$-tuples of $(k+1)$-ary trees with $n$ total nodes. Given Theorem \ref{t-kbox}, it follows that $k$-box paths of size $n$ are equinumerous with $(k+1)_{k}$-Dyck paths of size $n-1$.

There is a bijection between these paths which is easy to describe. Given a $k$-box path of size $n$, we delete the last $UD^{k}L$-factor and then replace each $UD^{k}LD$-factor with a $D^{(k+1)}$ step. The result is a prefix of a $(k+1)$-Dyck path that ends on the line $y=k$, and we obtain a $(k+1)_{k}$-Dyck path by vertically translating the path down $k$ units, deleting the first $k$ steps of the path, and then horizontally translating the path $k$ units to the left. See Figure \ref{f-selkirkbij} for an example. (This bijection does nothing when $k=0$, which makes sense because $0$-box paths of size $n$ and $1_{0}$-Dyck paths of size $n-1$ are both Dyck paths of semilength $n-1$.)

\begin{figure}
\noindent \begin{centering}
\begin{tikzpicture}[scale=0.4] 
\draw [line width=0] (4,2); 
\draw[pathcolorlight] (0,0) -- (17,0); 
\drawlinedots{0,1,2,3,4,5,6,7,6,7,8,9,10,11,12,11,12,13,14,15,16,17,16}{0,1,2,3,4,5,4,3,2,1,2,3,4,3,2,1,0,1,2,3,2,1,0}
\end{tikzpicture}
\hspace{0.5cm}
$\boldsymbol{\longleftrightarrow}$
\hspace{0.5cm}
\begin{tikzpicture}[scale=0.4] 
\draw [line width=0] (4,2); 
\draw[pathcolorlight] (0,0) -- (12,0); 
\drawlinedots{0,1,2,5,6,7,10,11,12}{0,1,2,-1,0,1,-2,-1,0}
\end{tikzpicture}
\par\end{centering}
\caption{\label{f-selkirkbij}A $2$-box path and the corresponding $3_{2}$-Dyck
path}
\end{figure}
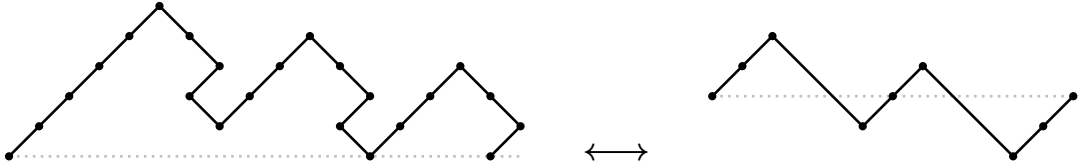

We note that there is a subfamily of Motzkin paths introduced in \cite{Prodinger2020}, called \textit{$T$-Motzkin paths}, which are in bijection with pairs of ternary trees. Selkirk gave a bijection between $T$-Motzkin paths and $2_{1}$-Dyck paths \cite[Section 8.2.1]{Selkirk2019}. The simplest bijection from $T$-Motzkin paths to box paths is perhaps obtained by composing Selkirk's bijection with the inverse of our above bijection for $k=1$.

\section{\label{s-rusu}A bijection between \texorpdfstring{$k$}{k}-box paths and \texorpdfstring{$(k+2,k)$}{(k+2)k}-threshold sequences}

Given integers $n\geq1$, $k\geq2$, and $0\leq\ell\leq k-2$, a \textit{$(k,\ell)$-threshold sequence} of \textit{length} $n$ is a strictly increasing sequence $(s_{1},s_{2},\dots,s_{n})$ of $n$ integers satisfying $ki\leq s_{i}\leq kn+\ell$ for each $i$. Threshold sequences were very recently defined and studied by Rusu \cite{Rusu2022}, who showed that $(k,\ell)$-threshold sequences of length $n$ are in bijection with ($\ell+1$)-tuples of $k$-ary trees with $n$ nodes and are therefore another member of the Fuss\textendash Catalan family of combinatorial objects. Below, we give a bijection between $k$-box paths of size $n$ and $(k+2,k$)-threshold sequences of length $n-1$.

Let $\mu=U^{a_{1}}D^{k}LD\cdots U^{a_{n-1}}D^{k}LDU^{a_{n}}D^{k}L$ be a $k$-box path of size $n$. We claim that the sequence $(s_{i})_{1\leq i\leq n-1}=(s_{1},s_{2},\dots,s_{n-1})$ obtained by taking $s_{i}=\sum_{j=1}^{i}a_{j}$ for each $1\leq i\leq n-1$ is a $(k+2,k)$-threshold sequence. Clearly this sequence is increasing as each $a_j$ is positive, so it suffices to verify that $(k+2)i\leq s_{i}\leq(k+2)(n-1)+k$ for each $i$. Since $\mu$ is a $k$-box path of size $n$, it has semilength $(k+2)n-1$; this is the number of up steps in $\mu$. Notice that the upper bound $(k+2)(n-1)+k$ is equal to $(k+2)n-2$, which is one less than the number of up steps. Because $a_{n}\geq1$ is not accounted for in any of the $s_{i}$, it follows that 
\[
s_{i}+1\leq s_{i}+a_{n}\leq(k+2)n-1
\]
for each $i$, and thus $s_{i}\leq(k+2)n-2$. For the lower bound, recall that each $UD^{k}L$-factor (except the last one) in $\mu$ is immediately followed by a $D$\textemdash hence, each of the $U^{a_{i}}$ for $1\leq i\leq n-1$ is immediately followed by a total of $k+2$ down and left steps before encountering another up step. So, if $a_{1}+a_{2}+\cdots+a_{i}=s_{i}$ were smaller than $(k+2)i$, then $\mu$ would go below the $x$-axis. This establishes the lower bound $s_{i}\geq(k+2)i$.

An enumerative combinatorialist will have recognized the map $\mu\mapsto(s_{i})_{1\leq i\leq n-1}$ given above as essentially (a restriction of) the classical bijection from compositions to subsets. As such, to recover $\mu$ from $(s_{i})_{1\leq i\leq n-1}$, we simply take $a_{i}=s_{i}-s_{i-1}$ for each $1\leq i\leq n$, where we set $s_{0}=0$ and $s_{n}=(k+2)n-1$. It is straightforward to show that this yields a $k$-box path of size $n$ given any $(k+2,k)$-threshold sequence of length $n-1$; we leave the details to the reader.

\section{Counting \texorpdfstring{$k$}{k}-box paths by returns} \label{sec-6}

Let us now turn our attention to the refined enumeration of $k$-box paths by statistics, beginning with the number of returns. Define $\grave{f}_{k,n,j}$ to be the number of $k$-box paths of size $n$ with exactly $j$ returns, and $\grave{g}_{k,n,j}$ the number of augmented $k$-Dyck paths of size $n$ with exactly $j$ returns. Then take
\[
\grave{F}_{k}(t,x)\coloneqq\sum_{n=1}^{\infty}\sum_{j=1}^{n}\grave{f}_{k,n,j}t^{j}x^{n}\quad\text{\text{and}}\quad\grave{G}_{k}(t,x)\coloneqq\sum_{n=0}^{\infty}\sum_{j=1}^{n}\grave{g}_{k,n,j}t^{j}x^{n}
\]
to be the bivariate generating functions for counting $k$-box paths and augmented $k$-Dyck paths, respectively, by the number of returns. Also let 
\[
G_{k}(x)\coloneqq\sum_{n=0}^{\infty}|\mathcal{\hat{D}}_{n}^{(k)}|x^{n}=\sum_{n=0}^{\infty}\frac{1}{(k+1)n+1}{(k+1)n+1 \choose n}x^{n}.
\]
be the ordinary generating function for augmented $k$-Dyck paths. For simplicity, let us write $\grave{F}_{k}=\grave{F}_{k}(t,x)$, $\grave{G}_{k}=\grave{G}_{k}(t,x)$, and $G_{k}=G_{k}(x)$.

The distribution of the number of returns over $k$-Dyck paths (also given by the numbers $\grave{g}_{k,n,j}$) was previously studied by Cameron and McLeod \cite{Cameron2016}. We need the generating function $\grave{G}_{k}$ in order to derive a formula for $\grave{f}_{k,n,j}$.

\begin{lem}
\label{l-retfeq}
We have
\[
\grave{F}_{k}=\frac{txG_{k+1}^{k}}{1-txG_{k+1}^{k+1}}.
\]
\end{lem}

\begin{proof}
Recall that every augmented $k$-Dyck path $\mu$ is either the empty path or can be uniquely decomposed as
\[
\mu=U\mu_{1}U\mu_{2}U\mu_{3}\cdots U\mu_{k}UD^{k-1}LD\mu_{k+1}
\]
where the $\mu_{1},\dots,\mu_{k+1}$ are augmented $k$-Dyck paths. In this decomposition, observe that the number of returns of $\mu$ is one more than the number of returns of $\mu_{k+1}$, and the number of $UD^{k-1}LD$-factors in $\mu$ is one more than the total number of such factors in the $\mu_{1},\dots,\mu_{k+1}$. Hence, we have the functional equation 
\[
\grave{G}_{k}=1+txG_{k}^{k}\grave{G}_{k},
\]
which gives us
\[
\grave{G}_{k}=\frac{1}{1-txG_{k}^{k}}.
\]
From the decomposition for $k$-box paths given in Lemma \ref{l-tboxdecomp} and similar reasoning, we see that 
\[
\grave{F}_{k}=txG_{k+1}^{k}\grave{G}_{k+1}=\frac{txG_{k+1}^{k}}{1-txG_{k+1}^{k+1}}
\]
as desired.
\end{proof}

\begin{thm}
\label{t-kboxret}
The number of $k$-box paths of size $n$ with exactly $j$ returns is equal to
\[
\grave{f}_{k,n,j}=\frac{j(k+1)-1}{n-j}{(k+2)n-j-2 \choose n-j-1}=\frac{j(k+1)-1}{(k+2)n-j-1}{(k+2)n-j-1 \choose n-j}.
\]
\end{thm}

\begin{proof}
Using Lemma \ref{l-retfeq} and the formula for Fuss\textendash Catalan numbers (which give the coefficients of $G_{k}^{r}$), we obtain
\begin{align*}
[t^{j}x^{n}]\,\grave{F}_{k} & =[t^{j}x^{n}]\,\frac{txG_{k+1}^{k}}{1-txG_{k+1}^{k+1}}\\
 & =[t^{j}x^{n}]\,\sum_{j=0}^{\infty}t^{j+1}x^{j+1}G_{k+1}^{j(k+1)+k}\\
 & =[t^{j}x^{n}]\,\sum_{j=1}^{\infty}t^{j}x^{j}G_{k+1}^{(j-1)(k+1)+k}\\
 & =[x^{n-j}]\,G_{k+1}^{j(k+1)-1}\\
 & =\frac{j(k+1)-1}{n-j}{(k+2)(n-j)+j(k+1)-2 \choose n-j-1}\\
 & =\frac{j(k+1)-1}{n-j}{(k+2)n-j-2 \choose n-j-1}\\
 & =\frac{j(k+1)-1}{(k+2)n-j-1}{(k+2)n-j-1 \choose n-j}.\qedhere
\end{align*}
\end{proof}
For $k=0$, we recover the formula 
\[
\grave{f}_{0,n+1,j+1}=\frac{j}{2n-j}{2n-j \choose n}
\]
for the number of Dyck paths of semilength $n$ with $j$ returns \cite[Section 6.6]{Deutsch1999}. The arrays of the numbers $\grave{f}_{k,n,j}$ for $k=1$ and $k=2$ are displayed in Tables \ref{tb-1ret}\textendash \ref{tb-2ret}. The array for $k=1$ can be found in entry A143603 of the OEIS \cite{oeis}.

\renewcommand{\arraystretch}{1.2}
\begin{table}[t]
\begin{centering}
\begin{tabular}{|c|>{\centering}p{1.2cm}|>{\centering}p{1.2cm}|>{\centering}p{1.2cm}|>{\centering}p{1.2cm}|>{\centering}p{1.2cm}|>{\centering}p{1.2cm}|>{\centering}p{1.2cm}|>{\centering}p{1.2cm}|}
\hline 
\diagbox[height=0.9cm]{$n$}{$j$} & $1$ & $2$ & $3$ & $4$ & $5$ & $6$ & $7$ & $8$\tabularnewline
\hline 
$1$ & $1$ &  &  &  &  &  &  & \tabularnewline
\hline 
$2$ & $1$ & $1$ &  &  &  &  &  & \tabularnewline
\hline 
$3$ & $3$ & $3$ & $1$ &  &  &  &  & \tabularnewline
\hline 
$4$ & $12$ & $12$ & $5$ & $1$ &  &  &  & \tabularnewline
\hline 
$5$ & $55$ & $55$ & $25$ & $7$ & $1$ &  &  & \tabularnewline
\hline 
$6$ & $273$ & $273$ & $130$ & $42$ & $9$ & $1$ &  & \tabularnewline
\hline 
$7$ & $1428$ & $1428$ & $700$ & $245$ & $63$ & $11$ & $1$ & \tabularnewline
\hline 
$8$ & $7752$ & $7752$ & $3876$ & $1428$ & $408$ & $88$ & $13$ & $1$\tabularnewline
\hline 
\end{tabular}
\par\end{centering}
\caption{\label{tb-1ret}The number of box paths of size $n$ with $j$ returns}
\end{table}
\begin{table}[t]
\begin{centering}
\begin{tabular}{|c|>{\centering}p{1.2cm}|>{\centering}p{1.2cm}|>{\centering}p{1.2cm}|>{\centering}p{1.2cm}|>{\centering}p{1.2cm}|>{\centering}p{1.2cm}|>{\centering}p{1.2cm}|>{\centering}p{1.2cm}|}
\hline 
\diagbox[height=0.9cm]{$n$}{$j$} & $1$ & $2$ & $3$ & $4$ & $5$ & $6$ & $7$ & $8$\tabularnewline
\hline 
$1$ & $1$ &  &  &  &  &  &  & \tabularnewline
\hline 
$2$ & $2$ & $1$ &  &  &  &  &  & \tabularnewline
\hline 
$3$ & $9$ & $5$ & $1$ &  &  &  &  & \tabularnewline
\hline 
$4$ & $52$ & $30$ & $8$ & $1$ &  &  &  & \tabularnewline
\hline 
$5$ & $340$ & $200$ & $60$ & $11$ & $1$ &  &  & \tabularnewline
\hline 
$6$ & $2394$ & $1425$ & $456$ & $99$ & $14$ & $1$ &  & \tabularnewline
\hline 
$7$ & $17710$ & $10626$ & $3542$ & $847$ & $147$ & $17$ & $1$ & \tabularnewline
\hline 
$8$ & $135720$ & $81900$ & $28080$ & $7150$ & $1400$ & $204$ & $20$ & $1$\tabularnewline
\hline 
\end{tabular}
\par\end{centering}
\caption{\label{tb-2ret}The number of 2-box paths of size $n$ with $j$ returns}
\end{table}

\begin{rem}
Notice that each row of Tables \ref{tb-1ret}\textendash \ref{tb-2ret} is (weakly) decreasing. This observation can be confirmed by calculating
\[
\grave{f}_{k,n,j}-\grave{f}_{k,n,j+1}=\frac{((k+1)j-2)((k+2)n-j-3)!}{((k+1)n-2)!(n-j)!},
\]
which is non-negative for all $k\geq1$, $n\geq2$, and $1\leq j\leq n-1$. There is also a simple combinatorial proof for this fact by way of an injection from paths in $\mathcal{B}_{n}^{(k)}$ with $j+1$ returns to those with $j$ returns. Given $\mu\in\mathcal{B}_{n}^{(k)}$ with $j+1$ returns, we can obtain a path $\mu^{\prime}\in\mathcal{B}_{n}^{(k)}$ with $j$ returns by deleting the first $U$ step after the first return and inserting it at the very beginning of the path. Then the position where $\mu$ had its first return is where $\mu^{\prime}$ has its first return to the line $y=1$, and so we can recover $\mu$ from $\mu^{\prime}$ by deleting the first $U$ step and inserting it immediately after the first return to $y=1$. See Figure \ref{f-retinj} for an example. Note that this reverse procedure is not well-defined in general, since the first return to $y=1$ might be in the middle of a $D^{k}LD$-factor or the ending $D^{k}L$-factor. (When $k=0$, this happens when our Dyck path $\mu^{\prime}$ begins with a $UD$-factor.) Thus, this injection is not a bijection.
\begin{figure}
\noindent \begin{centering}
\begin{tikzpicture}[scale=0.4] 
\draw[pathcolorlight] (0,0) -- (17,0); 
\drawlinedots{0,1,2,3,4,5,6,7,6,7,8,9,10,11,12,11,12,13,14,15,16,17,16}{0,1,2,3,4,5,4,3,2,1,2,3,4,3,2,1,0,1,2,3,2,1,0}
\end{tikzpicture}
\hspace{0.3cm}
$\boldsymbol{\longleftrightarrow}$
\hspace{0.3cm}
\begin{tikzpicture}[scale=0.4] 
\draw[pathcolorlight] (0,0) -- (17,0); 
\drawlinedots{0,1,2,3,4,5,6,7,8,7,8,9,10,11,12,13,12,13,14,15,16,17,16}{0,1,2,3,4,5,6,5,4,3,2,3,4,5,4,3,2,1,2,3,2,1,0}
\end{tikzpicture}
\par\end{centering}
\caption{\label{f-retinj}A $2$-box path with 2 returns and the corresponding
$2$-box path with 1 return}
\end{figure}
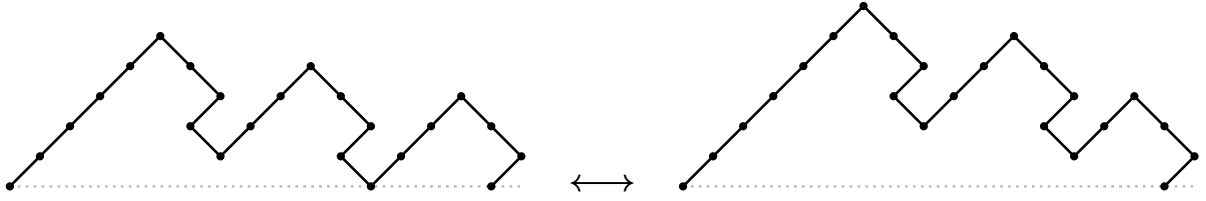

However, when $k=1$ and $j=1$, we claim that this injection is in fact a bijection. Let $\mu^{\prime}\in\mathcal{B}_{n}^{(1)}$ (for some $n\geq2$) have only 1 return. The first return to the line $y=1$ cannot be followed by a $L$ step, because then we are forced to have $\mu^{\prime}=UUDL$ which is of size $n=1$. Moreover, the first return to the line $y=1$ cannot be preceded by a $L$ step; otherwise, that $L$ step would be followed by a $D$ step which would end the path (since $\mu^{\prime}$ has only 1 return), yet box paths must end with an $L$. Therefore, the first return to $y=1$ is not in the middle of a $DLD$-factor or the ending $DL$-factor, so the reverse procedure described above is well-defined for all paths in $\mathcal{B}_{n}^{(1)}$ with 1 return. Hence $\grave{f}_{1,n,1}=\grave{f}_{1,n,2}$ for all $n\geq2$, as suggested by the $j=1$ and $j=2$ columns of Table \ref{tb-1ret}. Similarly, in the $k=0$ case, a Dyck path of semilength $n\geq2$ with only 1 return cannot begin with a $UD$-factor, so we have $\grave{f}_{0,n,2}=\grave{f}_{0,n,3}$ for all $n\geq3$.
\end{rem}

Recall that $f_{k,n}$ denotes the total number of $k$-box paths of size $n$. Using our formulas for $f_{k,n}$ and $\grave{f}_{k,n,j}$ from Theorems \ref{t-kbox} and \ref{t-kboxret}, together with the \texttt{sum} command in Maple, we were able to find closed-form formulas for the expressions
\[
\frac{\sum_{j=1}^{n}j\grave{f}_{k,n,j}}{f_{k,n}}\quad\text{and}\quad\frac{\sum_{j=1}^{n}j^{2}\grave{f}_{k,n,j}}{f_{k,n}}-\left(\frac{\sum_{j=1}^{n}j\grave{f}_{k,n,j}}{f_{k,n}}\right)^{2}
\]
giving the average value and variance, respectively, for the number of returns.

\begin{cor}
\label{c-retstats}
Among $k$-box paths of size $n$, the average number of returns is 
\[
\frac{(k+2)((k+2)n-1)}{(k+1)((k+1)n+1)}
\]
and the variance in the number of returns is
\[
\frac{(k+2)((k+2)n-1)((2k+1)(k+1)n-2)(n-1)}{(k+1)^{2}((k+1)n+1)^{2}((k+1)n+2)}.
\]
\end{cor}

As $n\rightarrow\infty$, the average tends to $(k+2)^{2}/(k+1)^{2}$ and the variance to $(2k+1)(k+2)^{2}/(k+1)^{4}$, which are approximately 1 and 0, respectively, when $k$ is large.

We note that the average number of returns of a $k$-box path is one more than the number of returns to the line $y=-k$ of the $(k+1)_{k}$-Dyck path obtained via the bijection from Section \ref{s-selkirk}. In fact, Corollary \ref{c-retstats} is equivalent to the mean and variance formulas (specialized appropriately) from Selkirk's Propositions 6.2.1 and 6.2.2 \cite{Selkirk2019}. (Selkirk did not derive an exact enumeration formula for $k_{t}$-Dyck paths with a given number of returns to $y=-t$.)

\section{Counting \texorpdfstring{$k$}{k}-box paths by long ascents} \label{sec-7}

As mentioned in the introduction, the number of peaks ($UD$-factors) is another statistic whose distribution over Dyck paths is well known. However, a $k$-box path of size $n$ has exactly $n$ peaks, so simply counting $k$-box paths by peaks does not yield anything new. Another parameter is the number of \textit{ascents}: maximal consecutive subsequences of up steps. Ascents and peaks coincide in Dyck paths and skew Dyck paths\textemdash every ascent ends in a peak and every peak is at the end of an ascent\textemdash although they have different distributions on some other families of lattice paths like Motzkin paths. Since it is also true that a $k$-box path of size $n$ has exactly $n$ ascents, let us instead count $k$-box paths by the number of \textit{long ascents}: ascents consisting of at least two up steps.

Let $\acute{f}_{k,n,j}$ be the number of $k$-box paths of size $n$ with exactly $j$ long ascents, and $\acute{g}_{k,n,j}$ the number of augmented $k$-Dyck paths of size $n$ with exactly $j$ long ascents. Then define the bivariate generating functions
\[
\acute{F}_{k}(t,x)\coloneqq\sum_{n=1}^{\infty}\sum_{j=1}^{n}\acute{f}_{k,n,j}t^{j}x^{n}\quad\text{\text{and}}\quad\acute{G}_{k}(t,x)\coloneqq\sum_{n=0}^{\infty}\sum_{j=1}^{n}\acute{g}_{k,n,j}t^{j}x^{n};
\]
we write $\acute{F}_{k}=\acute{F}_{k}(t,x)$ and $\acute{G}_{k}=\acute{G}_{k}(t,x)$ for simplicity.
\begin{lem}
We have the functional equations 
\begin{align}
\acute{F}_{k} & =x\acute{G}_{k+1}^{k}(\acute{G}_{k+1}-1+t)\label{e-feqF}
\end{align}
and 
\begin{equation}
\acute{G}_{k}=1+x\acute{G}_{k}^{k}(\acute{G}_{k}-1+t).\label{e-feqG}
\end{equation}
\end{lem}

\begin{proof}
We begin by proving (\ref{e-feqG}). Recall that every augmented $k$-Dyck path $\mu$ is either the empty path (which contributes the 1 on the right-hand side) or can be uniquely decomposed as
\[
\mu=U\mu_{1}U\mu_{2}U\mu_{3}\cdots U\mu_{k}UD^{k-1}LD\mu_{k+1}
\]
where the $\mu_{1},\dots,\mu_{k+1}$ are augmented $k$-Dyck paths. Observe the following:
\begin{itemize}
\item The number of $UD^{k-1}LD$-factors in $\mu$ is one more than the total number of such factors in the $\mu_{1},\dots,\mu_{k+1}$.
\item If $\mu_{k}$ is nonempty, then long ascents of $\mu$ are in correspondence with long ascents of the $\mu_{1},\dots,\mu_{k+1}$. Otherwise, if $\mu_{k}$ is empty, then the number of long ascents in $\mu$ is exactly one more than the total number of long ascents in the $\mu_{1},\dots,\mu_{k+1}$.
\end{itemize}
Thus we have (\ref{e-feqG}). Equation (\ref{e-feqF}) follows from the decomposition given by Lemma \ref{l-tboxdecomp} and similar reasoning.
\end{proof}

\begin{thm}
\label{t-lasc}
The number of $k$-box paths of size $n$ with exactly
$j$ long ascents is equal to
\[
\acute{f}_{k,n,j}=\frac{1}{j}{(k+1)n-2 \choose j-1}{n-1 \choose j-1}.
\]
\end{thm}

While we are mainly interested in the numbers $\acute{f}_{k,n,j}$, we note that a formula for $\acute{g}_{k,n,j}$ is obtained by specializing (\ref{e-rpower}) below at $r=1$.

\begin{proof}
Let $\tilde{G}_{k}\coloneqq\acute{G}_{k}-1$. Then substituting $\acute{G}_{k}=\tilde{G}_{k}+1$ into (\ref{e-feqG}) and simplifying gives us 
\[
\tilde{G}_{k}=x(\tilde{G}_{k}+1)^{k}(\tilde{G}_{k}+t)=xR(\tilde{G}_{k})
\]
where $R(u)=(u+1)^{k}(u+t)$. Let $\phi(u)=(u+1)^{r}$. Applying Lagrange inversion \cite{Gessel2016}, we obtain {\allowdisplaybreaks
\begin{align}
[t^{j}x^{n}]\,\acute{G}_{k}^{r} & =[t^{j}x^{n}]\,(\tilde{G}_{k}+1)^{r}\nonumber \\
 & =[t^{j}x^{n}]\,\phi(\tilde{G}_{k})\nonumber \\
 & =\frac{1}{n}\,[t^{j}u^{n-1}]\,\phi^{\prime}(u)R(u)^{n}\nonumber \\
 & =\frac{r}{n}\,[t^{j}u^{n-1}]\,(u+1)^{kn+r-1}(u+t)^{n}\nonumber \\
 & =\frac{r}{n}\,[t^{j}u^{n-1}]\,\sum_{i=0}^{kn+r-1}{kn+r-1 \choose i}u^{i}\sum_{l=0}^{n}{n \choose l}u^{n-l}t^{l}\nonumber \\
 & =\frac{r}{n}\,[t^{j}u^{n-1}]\,\sum_{l=0}^{n}\sum_{i=0}^{kn+r-1}{kn+r-1 \choose i}{n \choose l}t^{l}u^{n-l+i}\nonumber \\
 & =\frac{r}{n}\,[t^{j}u^{n-1}]\,\sum_{l=0}^{n}\sum_{m=n-l}^{(k+1)n+r-l-1}{kn+r-1 \choose m+l-n}{n \choose l}t^{l}u^{m}\nonumber \\
 & =\frac{r}{n}{kn+r-1 \choose j-1}{n \choose j}.\label{e-rpower}
\end{align}
}Now, from (\ref{e-feqF}), we have 
\begin{align*}
\acute{F}_{k} & =x\acute{G}_{k+1}^{k+1}-x\acute{G}_{k+1}^{k}+xt\acute{G}_{k+1}^{k}.
\end{align*}
Using (\ref{e-rpower}), we obtain
\begin{align*}
[t^{j}x^{n}]\,\acute{F}_{k} & =[t^{j}x^{n-1}]\,\acute{G}_{k+1}^{k+1}-[t^{j}x^{n-1}]\,\acute{G}_{k+1}^{k}+[t^{j-1}x^{n-1}]\,\acute{G}_{k+1}^{k}\\
 & =\frac{k+1}{n-1}{(k+1)n-1 \choose j-2}{n-1 \choose j-1}-\frac{k}{n-1}{(k+1)n-2 \choose j-2}{n-1 \choose j-1}\\
 & \qquad\qquad\qquad\qquad\qquad\qquad\qquad\qquad\qquad\qquad+\frac{k}{n-1}{(k+1)n-2 \choose j-1}{n-1 \choose j};
\end{align*}
this simplifies to the desired formula for $\acute{f}_{k,n,j}$ (as can be verified using Maple).
\end{proof}

Observe that the $\acute{f}_{k,n,j}$ specialized at $k=0$ are Narayana numbers; more precisely, we have $\acute{f}_{0,n,j}=N_{n-1,j}$ for all $n\geq2$ and $1\leq j\leq n$. This is because, in the $k=0$ case, all ascents are long ascents, and ascents in Dyck paths correspond directly to peaks (as noted earlier).

The arrays of the numbers $\acute{f}_{k,n,j}$ for $k=1$ and $k=2$ are displayed in Tables \ref{tb-1sasc}\textendash \ref{tb-2sasc}.
\begin{table}[t]
\begin{centering}
\begin{tabular}{|c|>{\centering}p{1cm}|>{\centering}p{1cm}|>{\centering}p{1cm}|>{\centering}p{1cm}|>{\centering}p{1cm}|>{\centering}p{1cm}|>{\centering}p{1cm}|>{\centering}p{1cm}|}
\hline 
\diagbox[height=0.9cm]{$n$}{$j$} & $1$ & $2$ & $3$ & $4$ & $5$ & $6$ & $7$ & $8$\tabularnewline
\hline 
$1$ & $1$ &  &  &  &  &  &  & \tabularnewline
\hline 
$2$ & $1$ & $1$ &  &  &  &  &  & \tabularnewline
\hline 
$3$ & $1$ & $4$ & $2$ &  &  &  &  & \tabularnewline
\hline 
$4$ & $1$ & $9$ & $15$ & $5$ &  &  &  & \tabularnewline
\hline 
$5$ & $1$ & $16$ & $56$ & $56$ & $14$ &  &  & \tabularnewline
\hline 
$6$ & $1$ & $25$ & $150$ & $300$ & $210$ & $42$ &  & \tabularnewline
\hline 
$7$ & $1$ & $36$ & $330$ & $1100$ & $1485$ & $792$ & $132$ & \tabularnewline
\hline 
$8$ & $1$ & $49$ & $637$ & $3185$ & $7007$ & $7007$ & $3003$ & $429$\tabularnewline
\hline 
\end{tabular}
\par\end{centering}
\caption{\label{tb-1sasc}The number of box paths of size $n$ with $j$ long
ascents}
\end{table}
\begin{table}[t]
\begin{centering}
\begin{tabular}{|c|>{\centering}p{1cm}|>{\centering}p{1cm}|>{\centering}p{1cm}|>{\centering}p{1cm}|>{\centering}p{1cm}|>{\centering}p{1cm}|>{\centering}p{1cm}|>{\centering}p{1cm}|}
\hline 
\diagbox[height=0.9cm]{$n$}{$j$} & $1$ & $2$ & $3$ & $4$ & $5$ & $6$ & $7$ & $8$\tabularnewline
\hline 
$1$ & $1$ &  &  &  &  &  &  & \tabularnewline
\hline 
$2$ & $1$ & $2$ &  &  &  &  &  & \tabularnewline
\hline 
$3$ & $1$ & $7$ & $7$ &  &  &  &  & \tabularnewline
\hline 
$4$ & $1$ & $15$ & $45$ & $30$ &  &  &  & \tabularnewline
\hline 
$5$ & $1$ & $26$ & $156$ & $286$ & $143$ &  &  & \tabularnewline
\hline 
$6$ & $1$ & $40$ & $400$ & $1400$ & $1820$ & $728$ &  & \tabularnewline
\hline 
$7$ & $1$ & $57$ & $855$ & $4845$ & $11628$ & $11628$ & $3876$ & \tabularnewline
\hline 
$8$ & $1$ & $77$ & $1617$ & $13475$ & $51205$ & $92169$ & $74613$ & $21318$\tabularnewline
\hline 
\end{tabular}
\par\end{centering}
\caption{\label{tb-2sasc}The number of 2-box paths of size $n$ with $j$
long ascents}
\end{table}

\begin{rem}
Let us record and discuss a few observations from Tables \ref{tb-1sasc}\textendash \ref{tb-2sasc}.
\begin{itemize}
\item The main diagonal of Table \ref{tb-1sasc} consists of Catalan numbers and that of Table \ref{tb-2sasc} enumerates box paths. In general, the numbers
\[
\acute{f}_{k+1,n,n}=\frac{1}{n}{(k+2)n-2 \choose n-1}=\frac{1}{(k+2)n-1}{(k+2)n-1 \choose n}=f_{k,n}
\]
count $k$-box paths of size $n$. In fact, there is a simple bijection between $k$-box paths of size $n$ and $(k+1)$-box paths of size $n$ with all ascents long: we replace every $UD^{k}L$-factor with a $UUD^{k+1}L$-factor. (If $k=0$, then we replace each $D$ with a $UUDLD$ and then append a $UUDL$ to the end of the path.)
\item The $j=2$ column of Table \ref{tb-1sasc} consists of perfect squares and that of Table \ref{tb-2sasc} consists of second pentagonal numbers. These are the $k=4,5$ cases of the \textit{second $k$-gonal numbers} \cite[Section 1.8]{Deza2012}, which are defined by taking the formula
\[
\frac{n}{2}((k-2)n-(k-4))
\]
for the ordinary $k$-gonal numbers and replacing $n$ with $-n$. More generally, we have 
\begin{align*}
\acute{f}_{k,n+1,2} & =\frac{1}{2}((k+1)(n+1)-2)n\\
 & =\frac{n}{2}((k+3)-4+((k+3)-2)n)
\end{align*}
which forms the sequence of second $(k+3)$-gonal numbers. The second $k$-gonal numbers have been well studied from the perspective of number theory, but we are not aware of any previous combinatorial interpretation for these numbers that are valid for general values of $k$. Hence, this result seems interesting enough for us to state it as a theorem below.
\item Table \ref{tb-1sasc} consists of three pairs of consecutive repeating numbers, each three rows apart: $(1,1)$, $(56,56)$, and $(7007,7007)$. Similarly, Table \ref{tb-2sasc} has $(7,7)$ and $(11628,11628)$, which are four rows apart. In general, we have 
\[
\acute{f}_{k,(k+2)i-1,(k+1)i-1}=\acute{f}_{k,(k+2)i-1,(k+1)i}
\]
for all $k,i\geq1$, as can be checked using Theorem \ref{t-lasc}. Is there a bijective proof?
\end{itemize}
\end{rem}

\begin{thm}
For all $k\geq3$ and $n\geq1$, the $n$th second $k$-gonal number is equal to the number $\acute{f}_{k-3,n+1,2}$ of $(k-3)$-box paths of size $n+1$ with exactly two long ascents.
\end{thm}

\begin{rem}
It can be verified that 
\begin{multline*}
\quad(\acute{f}_{k,n,j})^{2}-\acute{f}_{k,n,j-1}\acute{f}_{k,n,j+1}\\
=\frac{n(2(k+1)n+k-(k+2)j)}{(n-j+1)j^{2}(j+1)((k+1)n+j)}{(k+1)n-2 \choose j-1}^{2}{n-1 \choose j-1}^{2}>0\quad
\end{multline*}
for all $n\geq2$, $k\geq0$, and $1<j<n$. In other words, for fixed $n$ and $k$, the sequences $\{\acute{f}_{k,n,j}\}_{0\leq j<n}$ are log-concave and thus unimodal. It would be interesting to find a combinatorial proof for log-concavity via constructing an appropriate injection.
\end{rem}

As with the case with returns, we can use our formulas for $\acute{f}_{k,n,j}$ and $f_{k,n}$ along with the \texttt{sum} command in Maple to obtain a closed expression for the average number of long ascents. However, we were unable to use Maple to evaluate the summation $\sum_{j=1}^{n}j^{2}\acute{f}_{k,n,j}$ needed for the variance. Instead, we evaluate this summation directly with the help of Vandermonde's identity 
\[
\sum_{j=0}^{r}{p \choose j}{q \choose r-j}={p+q \choose r}.
\]
This approach also allows us to evaluate $\sum_{j=1}^{n}j\acute{f}_{k,n,j}$ which is needed for the mean, and so we demonstrate both for the sake of completeness.

\begin{lem}
\label{l-lascsums}
We have 
\[
\sum_{j=1}^{n}j\acute{f}_{k,n,j}={(k+2)n-3 \choose n-1}\quad\text{and}\quad\sum_{j=1}^{n}j^{2}\acute{f}_{k,n,j}=\frac{(k+1)n^{2}-n-1}{(k+2)n-3}{(k+2)n-3 \choose n-1}.
\]
\end{lem}

\begin{proof}
First, observe that setting $p=(k+1)n-2$ and $q=r=n-1$ in Vandermonde's identity yields 
\begin{equation}
\sum_{j=0}^{n-1}{(k+1)n-2 \choose j}{n-1 \choose j}={(k+2)n-3 \choose n-1}.\label{e-vspec}
\end{equation}
We also have
\begin{align}
\sum_{j=0}^{n-1}j{(k+1)n-2 \choose j}{n-1 \choose j} & =((k+1)n-2)\sum_{j=0}^{n-1}{(k+1)n-3 \choose j-1}{n-1 \choose j}\nonumber \\
 & =((k+1)n-2)\sum_{j=0}^{n-2}{(k+1)n-3 \choose j}{n-1 \choose j+1}\nonumber \\
 & =((k+1)n-2)\sum_{j=0}^{n-2}{(k+1)n-3 \choose j}{n-1 \choose n-j-2}\nonumber \\
 & =((k+1)n-2){(k+2)n-4 \choose n-2}\label{e-vspec2}
\end{align}
where the last equality is Vandermonde's identity at $p=(k+1)n-3$, $q=n-1$, and $r=n-2$. Using (\ref{e-vspec}), we obtain
\begin{align*}
\sum_{j=1}^{n}j\acute{f}_{k,n,j} & =\sum_{j=1}^{n}{(k+1)n-2 \choose j-1}{n-1 \choose j-1}\\
 & =\sum_{j=0}^{n-1}{(k+1)n-2 \choose j}{n-1 \choose j}\\
 & ={(k+2)n-3 \choose n-1},
\end{align*}
and using both (\ref{e-vspec}) and (\ref{e-vspec2}), we have {\allowdisplaybreaks
\begin{align*}
\sum_{j=1}^{n}j^{2}\acute{f}_{k,n,j} & =\sum_{j=1}^{n}j{(k+1)n-2 \choose j-1}{n-1 \choose j-1}\\
 & =\sum_{j=0}^{n-1}(j+1){(k+1)n-2 \choose j}{n-1 \choose j}\\
 & =\sum_{j=0}^{n-1}j{(k+1)n-2 \choose j}{n-1 \choose j}+\sum_{j=0}^{n-1}{(k+1)n-2 \choose j}{n-1 \choose j}\\
 & =((k+1)n-2){(k+2)n-4 \choose n-2}+{(k+2)n-3 \choose n-1}\\
 & =\frac{(k+1)n^{2}-n-1}{(k+2)n-3}{(k+2)n-3 \choose n-1}.\qedhere
\end{align*}
}
\end{proof}

Lemma \ref{l-lascsums} and Theorem \ref{t-kbox} lead immediately to the following.

\begin{prop}
\label{p-lascstats}
Among $k$-box paths of size $n$, the average number of long ascents is
\[
\frac{n((k+1)n-1)}{(k+2)n-2}
\]
and the variance in the number of long ascents is 
\[
{\displaystyle \frac{((k+1)n-1)((k+1)n-2)n(n-1)}{((k+2)n-2)^{2}((k+2)n-3)}}.
\]
\end{prop}

For large $n$, the expected number of long ascents is approximately $(k+1)n/(k+2)$ and the variance is approximately $(k+1)^{2}n/(k+2)^{3}$.

Returning to the connection between $k$-box paths and $k_{t}$-Dyck paths, let $\mu$ be a $k$-box path and $\mu^{\prime}$ the associated $(k+1)_{k}$-Dyck path under the bijection from Section \ref{s-selkirk}. If the final ascent of $\mu$ is long, then the number of long ascents of $\mu$ is one more than the number of peaks of $\mu^{\prime}$. Otherwise, if the final ascent of $\mu$ has length 1, then the number of long ascents of $\mu$ is equal to the number of peaks of $\mu^{\prime}$. Selkirk gives mean and variance formulas for the number of peaks in $k_{t}$-Dyck paths, which does not translate directly to Proposition \ref{p-lascstats} but come close, due to the slight difference between peaks in $(k+1)_{k}$-Dyck paths and long ascents in $k$-box paths.

\vspace{10bp}

\noindent \textbf{Acknowledgments.} The work in this paper was initiated as part of the first author's honors thesis project at Davidson College (under the supervision of the second author); we thank Heather Smith Blake, Claire Merriman, Jonad Pulaj, and Carl Yerger for serving on the first author's thesis committee and giving feedback on his work. We also thank Ira Gessel for pointing out the applicability of Vandermonde's identity in proving Lemma \ref{l-lascsums}, Sarah Selkirk for her generous feedback on an earlier draft of this manuscript, and an anonymous referee for their suggestions. The second author was partially supported by an AMS-Simons Travel Grant and NSF grant DMS-2316181.

\bibliographystyle{plain}
\addcontentsline{toc}{section}{\refname}\bibliography{bibliography}

\providecommand\noopsort[1]{}
\begin{thebibliography}{10}

\bibitem{Asinowski2022}
Andrei Asinowski, Benjamin Hackl, and Sarah~J. Selkirk.
\newblock Down-step statistics in generalized {D}yck paths.
\newblock {\em Discrete Math. Theor. Comput. Sci.}, 24(1):Paper No. 17, 40~pp.,
  2022.

\bibitem{Cameron2016}
Naiomi~T. Cameron and Jillian~E. McLeod.
\newblock Returns and hills on generalized {D}yck paths.
\newblock {\em J. Integer Seq.}, 19(6):Article 16.6.1, 28~pp., 2016.

\bibitem{Deutsch1999}
Emeric Deutsch.
\newblock Dyck path enumeration.
\newblock {\em Discrete Math.}, 204(1-3):167--202, 1999.

\bibitem{Deutsch2010}
Emeric Deutsch, Emanuele Munarini, and Simone Rinaldi.
\newblock Skew {D}yck paths.
\newblock {\em J. Statist. Plann. Inference}, 140(8):2191--2203, 2010.

\bibitem{Deza2012}
Elena Deza and Michel~Marie Deza.
\newblock {\em Figurate numbers}.
\newblock World Scientific Publishing Co. Pte. Ltd., Hackensack, NJ, 2012.

\bibitem{Gessel2016}
Ira~M. Gessel.
\newblock Lagrange inversion.
\newblock {\em J. Combin. Theory Ser. A}, 144:212--249, 2016.

\bibitem{Heuberger2023}
Clemens Heuberger, Sarah~J. Selkirk, and Stephan Wagner.
\newblock Enumeration of generalized {D}yck paths based on the height of
  down-steps modulo {$k$}.
\newblock {\em Electron. J. Combin.}, 30(1):Paper No. 1.26, 18 pp., 2023.

\bibitem{Prodinger}
Helmut Prodinger.
\newblock Skew {D}yck paths without up-down-left.
\newblock {\em Art Discrete Appl. Math.}, 7(1):Paper No. 1.01, 9 pp., 2024.

\bibitem{Prodinger2020}
Helmut Prodinger, Sarah~J. Selkirk, and Stephan Wagner.
\newblock On two subclasses of {M}otzkin paths and their relation to ternary
  trees.
\newblock In {\em Algorithmic {C}ombinatorics: {E}numerative {C}ombinatorics,
  {S}pecial {F}unctions and {C}omputer {A}lgebra---{I}n {H}onour of {P}eter
  {P}aule on his 60th {B}irthday}, Texts Monogr. Symbol. Comput., pages
  297--316. Springer, Cham, 2020.

\bibitem{Rusu2022}
Irena Rusu.
\newblock Raney numbers, threshold sequences and {M}otzkin-like paths.
\newblock {\em Discrete Math.}, 345(11):Paper No. 113065, 8 pp., 2022.

\bibitem{Selkirk2019}
Sarah~Jane Selkirk.
\newblock On a generalisation of $k$-{D}yck paths.
\newblock Master's thesis, Stellenbosch University, 2019.

\bibitem{oeis}
N.~J.~A. Sloane.
\newblock The {O}n-{L}ine {E}ncyclopedia of {I}nteger {S}equences.
\newblock Published electronically at \url{https://oeis.org}.

\end{thebibliography}

\end{document}